\documentclass{amsart}

\usepackage[utf8]{inputenc}
\usepackage{amsmath, amssymb,amsthm,tikz}
\usepackage{tikz-cd}
\usepackage{tikz}
\usepackage{todonotes}
\usepackage{mathrsfs}
\usepackage{extarrows}
\usepackage{graphicx}
\usepackage{thmtools}
\usepackage{mathtools}
\usepackage{hyperref}

\usepackage[english]{babel}

\usepackage[toc,page]{appendix}

\newcommand{\N}{\mathbb{N}}
\newcommand{\Z}{\mathbb{Z}}

\newcommand{\C}{\mathbb{C}}
\newcommand{\Q}{\mathbb{Q}}

\newcommand{\Ff}{\mathcal{F}}
\newcommand{\Ee}{\mathcal{E}}

\newcommand{\Oo}{\mathcal{O}}

\newcommand{\Image}{\operatorname{Im}}

\newcommand{\Id}{\operatorname{Id}}

\newcommand{\ad}{\operatorname{ad}}

\newcommand{\Spec}{\operatorname{Spec}}

\newcommand{\Gal}{\operatorname{Gal}}

\newcommand{\GL}{\operatorname{GL}}

\hypersetup{
    colorlinks,
    citecolor=blue,
    filecolor=blue,
    linkcolor=blue,
    urlcolor=blue
}

\renewcommand{\tilde}{\widetilde}

\newenvironment{subproof}[1][\proofname]{%
  \begin{proof}[#1]%
}{%
  \end{proof}%
}

\title{$p$-integrability}
\author{Yujie Xu}
\address{Department of Mathematics, Columbia University}
\curraddr{}
\email{yujiexu@alumni.harvard.edu}
\date{September 12th, 2021}

\newtheorem{theorem}[equation]{Theorem}
\newtheorem{prop}[equation]{Proposition}

\newtheorem{lem}[equation]{Lemma}
\newtheorem{Coro}[equation]{Corollary}

\theoremstyle{definition}
\newtheorem{Defn}[equation]{Definition}
\newtheorem{remark}[equation]{Remark}
\newtheorem{numberedparagraph}[equation]{}
\newtheorem{conj}[equation]{Conjecture}

\begin{document}

\begin{abstract}
    In this short note, we prove the equivalence of Grothendieck-Katz $p$-curvature Conjecture with Conjecture F in \cite{Ekedahl-Shepherd-Barron-Taylor}. More precisely, we show that Conjecture F implies the $p$-curvature conjecture, and that the $p$-curvature Conjecture implies Conjecture F for the foliation attached to a vector bundle with integrable connection. 
\end{abstract}

\maketitle

\section{Background}
The existence of compact leaves of algebraic foliations has been instrumental in constructing algebraic subvarieties. In general, if a manifold is compact complex algebraic, and if integrable distributions on this manifold is given by an analytic sub-bundle, then any compact leaf would be a closed algebraic subvariety. 

Let $X$ be a normal complex quasi-projective variety with a given possibly singular integrable distribution on it, i.e. a coherent subsheaf
(in the Zariski topology) $\mathcal{F}$ of $T_X$ that is closed under the Lie algebra bracket. 
We say that the distribution $\mathcal{F}$ is \textit{algebraically integrable} if there exists a dense open subvariety $U\subset X$, with a smooth morphism $f:U\to V$ such that $\ker df=\mathcal{F}|_U$. 

Let $R$ be some finitely generated $\Z$-subalgebra of $\C$ over which $X$ and $\mathcal{F}$ have models. By \textit{reduction modulo almost all primes $p$}, we mean reduction modulo all maximal ideals of $R$ outside some proper closed subscheme of $\Spec R$. 

Let $Y$ be a normal variety over a field of characteristic $p$ and $\mathcal{G}$ be a coherent Lie subalgebra of $T_Y$, which we view as the sheaf of derivations of $\Oo_Y$. Since the base field has characteristic $p$, the $p$-th power of a derivation is a derivation. We say that $\mathcal{G}$ is \textit{$p$-integrable} if it is closed under taking $p$-th powers. 
\begin{remark}
If $\mathcal{G}$ is algebraically integrable, it is in particular $p$-integrable. The converse, however, is not true. 
\end{remark}
Likewise, we say that an integrable connection $\nabla$ is \textit{$p$-integrable} if 
\[\nabla(D^p)-\nabla(D)^p=0.\] 

\begin{conj}\label{Conj-F}(Conjecture F \cite{Ekedahl-Shepherd-Barron-Taylor}) $\mathcal{F}$ is algebraically integrable $\Longleftrightarrow$ the reduction of $\mathcal{F}$ is $p$-integrable modulo almost all primes $p$.  
\end{conj}

\begin{conj}\label{p-curvature-conj}(Grothendieck-Katz $p$-curvature Conjecture \cite{Katz-pcurvature-Invent}) An integrable connection has a finite monodromy group $\Longleftrightarrow$ it is $p$-integrable modulo almost all primes $p$. 
\end{conj}

\section{Proof of the equivalence of conjectures}
\begin{numberedparagraph}
First note that algebraic integrability of a foliation implies the algebraicity of leaves. In order to prove the equivalence of Conjecture F for the foliation $\Ff_H$ associated to a vector bundle $\Ee$ with integrable connection $H$, with the Grothendieck-Katz $p$-curvature Conjecture, first we prove the following statement:
\begin{prop}\label{algebraic-integrable-finite-monodromy}$\Ff_H$ is algebraically integrable $\Leftrightarrow (\Ee, H)$ has finite monodromy. 
\begin{proof}
``$\Leftarrow$'': Suppose we have a vector bundle $\Ee\to X$ with flat connection $\nabla$, which has finite monodromy, i.e. $\Image(\pi_1(X)\xrightarrow{\varphi}\GL_n)$ is finite, i.e. if we denote $K=\Image(\varphi)$ which is a finite group, we can consider the surjective map $f: \pi_1(X)\twoheadrightarrow K=\Image(\varphi)$. By the classification of covering spaces, there exists a finite \'etale Galois cover $\tilde{X}\to X$ corresponding to the group $K$ where we have $K=\pi_1(X)/\pi_1(\tilde{X})$, i.e. $\pi_1(\tilde{X})=\ker f$. 
Since finite maps are projective and $X$ is projective, $\tilde{X}$ is projective as well. We can pullback the vector bundle $\Ee$ along $\tilde{X}\to X$ and obtain the pullback bundle $\tilde{\Ee}$ over $\tilde{X}$.

By our construction, 
$\tilde{X}$ has trivial monodromy (since it's the kernel of $f$, it gets mapped into zero in the monodromy map). Recall that by the Riemann-Hilbert correspondence, $(\tilde{\Ee},\tilde{\nabla})$ gives a local system on $\tilde{X}$, which is projective with trivial monodromy, thus by Riemann-Hilbert again, $(\tilde{\Ee},\tilde{\nabla})$ must be trivial. Thus we must have $\Ee\cong X\times V\xrightarrow{f}V$, and thus we can use Galois descent
, by the $G$-equivariance of $\tilde{\Ee}$ and obtain the following diagram 
\[\begin{tikzcd}
\tilde{\Ee} \arrow{r}{} \arrow{d}{} 
&V \arrow{d}{}\\
\Ee \arrow{r}{} & V/G
\end{tikzcd}\]
And we get a ``leaf space'' $V/G$ which parametrizes the leaf space for $\Ee$. Thus by definition $\Ee$ is algebraically integrable.

``$\Rightarrow$'': Now suppose $\Ff_H$ is algebraically integrable, this immediately implies that the leaves (which are fibres of a smooth map) are algebraic. To show that the monodromy is finite, by Lemma \ref{finite-orbit-lemma} below, it suffices to show that every orbit is finite. 
Since the leaves are algebraic, they're finite over the base (i.e. the map from the leaves to the base is finite), i.e. the leaves are finite fibres of the projection onto the base, and since the fibres contain the orbits of monodromy, we have that all orbits of the monodromy are finite. Thus by the following Lemma \ref{finite-orbit-lemma}, the monodromy is finite. 
\begin{lem}\label{finite-orbit-lemma}
If every orbit of $G\subset \GL_n(\C)$ is finite, then $G$ is finite. 
\end{lem} 
\begin{subproof}
We prove by contradiction. Suppose $G$ is not finite, then we can choose a countably infinite sequence $g_1, g_2, \cdots\in G\leq \GL_n(\C)$. First we claim that there is a vector $v$ s.t. $g_1v\neq g_2v\neq \cdots$ where all $g_iv$'s are pairwise distinct. To see this, we consider the sets 
$S_{i,j}:=\{v\in \C^n: g_iv=g_jv\}$, which are Zariski closed proper subvarieties of $\C^n$ (with the usual Zariski topology). Note that properness follows from the fact that, if $S_{i,j}=\C^n$, then we would have $g_i=g_j$ because linear operators on vector spaces are determined by their image. Now, we only need to show that the union of countably many proper closed subvariety is strictly included into $\C^n$, i.e. 
$\bigcup\limits_{i,j\in\N}S_{i,j}\subsetneq \C^n$, then we can take such a $0\neq v\in \C^n\setminus \bigcup\limits_{i,j\in\N}S_{i,j}$ to be our desired vector $v$.

One way to see the strict inclusion is to note that each $S_{i,j}$ has measure $0$ (here we consider the Lebesgue measure on $\C^n$), thus we have $m(\bigcup\limits_{i,j\in\N}S_{i,j})=0$ as the countable sum of zeros, whereas $m(\C^n)=+\infty\neq 0$, thus we must have a strict inclusion.\footnote{Alternatively, we can choose an affine embedding $\C\hookrightarrow \C^n$ s.t. $\C\cap S_{i,j}$ is a finite set of points, and thus $\C\cap(\bigcup\limits_{i,j\in\N}S_{i,j})$ is a countable set of points, however $\C\cap \C^n$ is an uncountable set of points, and thus we must have our desired strict inclusion.}
Now we can simply take $v\in \C^n\setminus \bigcup\limits_{i,j\in\N}S_{i,j}$, and then the orbit of $v$ is not finite, we get a contradiction. Thus $G$ must be finite. 
\end{subproof}
Thus we have proved our proposition for $\C$, it suffices to note that algebraic integrability is preserved under base change. Algebraic integrability over $\Q$ gives algebraic integrability over $\C$.

To go from algebraic integrability over $\C$ to algebraic integrability over $\Q$: first, given $U_{\C}$, if $U_{\C}$ is the set as in the definition of algebraic integrability, then for $\sigma\in\Gal(\C/\Q)$, $\sigma U_{\C}$ is also algebraically integrable, and we can define the corresponding $U_{\Q}$ as $U_{\Q}:=\bigcup\limits_{\sigma}\sigma U_{\C}$ which is also algebraically integrable. Note that by definition, the leaf space $V$ is unique.\footnote{Intuitively, $V$ is the moduli space of leaves, or we can view it as the categorical quotient of the sequence 
$U\times_V U\rightrightarrows U\to V$ 
where $U\times_V U=\{(x,y)|x,y\mbox{ are in the same leaf}\}$ and $V:=U\times _V U/U$ is the categorical quotient.} Thus by uniqueness we can obtain a unique $V_{\Q}$ and thus obtain algebraic integrability over $\Q$.  
\end{proof}
\end{prop}
\end{numberedparagraph}

\begin{numberedparagraph}
To finish the proof for the equivalence of the two conjectures, it only remains to show that $p$-integrability for a flat connection is equivalent to $p$-integrability for its associated foliation.

First we clarify the construction of an associated foliation from a given connection. Recall that a connection on a vector bundle $\pi: E\to B$ is given by a splitting of the short exact sequence
\begin{equation}\label{splitting-SES}
0\to V\to TE\to \pi^* TB\to 0 
\end{equation}
where $V=\ker d\pi$ is called the vertical bundle, and the splitting map 
\begin{equation}\label{splitting-map-varphi}
\varphi: \pi^* TB\to TE
\end{equation}
gives a decomposition into the vertical bundle and the horizontal bundle
\[TE=V\oplus H\]

\begin{Defn}
Given a connection $\varphi$, i.e. a splitting map of $(*)$, the foliation $\Ff_{\varphi}$ associated to this connection $\varphi$ is simply given by the image under this splitting map, i.e. we have
\[\Ff_{\varphi}:=\varphi(\pi^* TB)\]
\end{Defn}

Recall also that the $p$-integrability of the foliation $\Ff_{\varphi}$ is defined as the stability of $\Ff_{\varphi}$ under the $p$-th power.

Note that ``$p$-integrability for a connection'' as in \cite{Ekedahl-Shepherd-Barron-Taylor} is defined in terms of $\nabla$, i.e. an integrable connection $\nabla$ is $p$-integrable $\Longleftrightarrow \nabla(D^p)-\nabla(D)^p=0$. 
The natural question thus becomes: how do we express $\varphi$ in terms of $\nabla$ explicitly? We have two different ways of describing a connection, the more traditional way is as in Katz' papers \cite{Katz-pcurvature-Invent,Katz-conjecture} denoted as $\nabla$, and the alternative way is to describe it as a splitting map $\varphi$ as given in \ref{splitting-map-varphi}.  Geometrically, it's easy to see how these two definitions are related, because a connection $\nabla$ would give  
a horizontal subbundle by parallel transport, and thus identify a splitting map $\varphi$ whose image is this horizontal subbundle given by $\nabla$. However, to achieve our goal (i.e. proving the equivalence of Conjectures \ref{Conj-F} and \ref{p-curvature-conj}), we need to write down an explicit expression relating $\nabla$ and $\varphi$. 
To find such a relation, we take a given splitting map $\varphi$ as in equation \ref{splitting-map-varphi}, and show that the $\nabla$ defined as follows (in Proposition \ref{connection-splitting}) is indeed the $\nabla$  corresponding to $\varphi$.
\begin{prop}\label{connection-splitting} Given a splitting map $\varphi$ of the exact sequence \ref{splitting-SES}, the $\nabla$ as defined below is indeed the connection corresponding to this splitting $\varphi$,
\[\pi^*(\nabla_Ds)=[\varphi\pi^*D,\pi^*s]\]
 i.e. $\nabla$ satisfies Leibniz rule and linearity in $D$.

\begin{proof}
To check that $\nabla$ is indeed a connection, first we check the Leibniz rule: 
\begin{align*}
\nabla_D(fs)&=(\pi^*)^{-1}[\varphi\pi^*D,\pi^*(fs)]\\
&=(\pi^*)^{-1}[\varphi\pi^*D,(fs)\circ\pi]\\
&=(\pi^*)^{-1}[\varphi\pi^*D,(f\circ\pi)\cdot(\pi^*s)]\\
&=(\pi^*)^{-1}((\varphi\pi^*D)(f\circ\pi))\cdot(\pi^*s)+(\pi^*)^{-1}(f\circ \pi)[\varphi\pi^*D,\pi^*s]\\
&=(\pi^*)^{-1}((\varphi\pi^*D)(f\circ\pi)\cdot(\pi^*s))+((\pi^*)^{-1})(\pi^*f)\nabla_Ds
\end{align*}
Thus it suffices to show that the first term:
\[(\pi^*)^{-1}((\varphi\pi^*D)(f\circ\pi))=D(f)\]
Noting that $d\pi\circ \varphi=\Id$, we have
\begin{align*}
(\pi^*)^{-1}\Big((\varphi\pi^*D)(f\circ\pi)\Big)&=(\pi^*)^{-1}\Big((d\pi)^{-1}(d\pi\circ\varphi)(\pi^*D)(f\circ\pi)\Big)\\
&=(\pi^*)^{-1}\Big((d\pi)^{-1}(\pi^*D)(f\circ\pi)\Big)\\
&=(\pi^*)^{-1}\Big((\pi_*)^{-1}(\pi^*D)(f\circ\pi)\Big)
\end{align*}
Also note that 
\[(\pi^*D)(f\circ\pi)=(\pi^*D)(\pi^*f)=\pi_*(\pi^*D)(f)\]
Thus we plug it in and get
\[(\pi^*)^{-1}\Big((\varphi\pi^*D)(f\circ\pi)\Big)=D(f)\]
Thus we have the Leibniz rule
\[\nabla_D(fs)=D(f)s+f\nabla_Ds\]
To check that $\nabla$ is also linear in $D$: we want to compute
\[\nabla_{fD}s=(\pi^*)^{-1}[\varphi\pi^*(fD),\pi^*s]\]
First note that we have
\[\pi^*(fD)=(fD)(d\pi)\]
Pointwise it gives $\pi^*(fD)(p)=f(p)\cdot D(d\pi)(p)$, i.e. we have
\[\pi^*(fD)=f\pi^*(D)\]
Thus we have
\[\varphi\pi^*(fD)=\varphi(f\pi^*(D))\]
To show that $\varphi\pi^*(fD)=f\varphi\pi^*D$, it suffices to show that 
\[d\pi (f\varphi\pi^*D)=\pi^*(fD)\]
By linearity, the right hand side is
\[d\pi (f\varphi\pi^*D)=f d\pi(\varphi\pi^*D)=f\pi^*D=\pi^*(fD)\]
Thus $\varphi\pi^*(fD)=f\varphi\pi^*D$, and thus we have
\begin{align*}
\nabla_{fD}s&=(\pi^*)^{-1}[\varphi\pi^*(fD),\pi^*s]\\
&=(\pi^*)^{-1}[f\varphi\pi^*D,\pi^*s]\\
&=f(\pi^*)^{-1}[\varphi\pi^*D,\pi^*s]-(\pi^*)^{-1}\pi^*s(f)\cdot\varphi\pi^*D
\end{align*}
Note that since $\pi^*s\in V$, the second term evaluated at a point vanishes, thus we have 
\[\nabla_{fD}s=f\nabla_Ds\]
Thus the above defined $\nabla$ is indeed a connection, and it corresponds to the splitting map $\varphi$.
\end{proof}
\end{prop}

Since the derivations form a restricted Lie algebra, i.e. we have
\[\ad((\varphi\pi^*D)^p)=\ad(\varphi\pi^*D)^p\]
Thus we have the following Lemma.
\begin{lem}\label{p-power}Notations as above. 
\[(\nabla_D)^ps=\overbrace{[\varphi\pi^*D[\varphi\pi^*D[\varphi\pi^*D\cdots [\varphi\pi^*D,\pi^*s]]]]}^{p-times}=[(\varphi\pi^*D)^p,\pi^*s]\]
\end{lem}

Now we are ready to establish the equivalence between $p$-integrability for a connection $\nabla$ and $p$-integrability for its associated foliation. 
\begin{theorem}\label{p-integrability-equiv-connections-vs-foliation}
For an integrable connection $\nabla$, which comes from a splitting map $\varphi$ of the short exact sequence \ref{splitting-SES} as defined in Proposition \ref{connection-splitting}, the notion of $p$-integrability of the connection $\nabla$ is equivalent to the notion of $p$-integrability for the associated foliation $\Ff_{\varphi}$.
\end{theorem} 
\begin{proof}
``$\nabla$ is $p$-integrable $\Rightarrow$ foliation $\Ff_{\varphi}$ is $p$-integrable'':

Suppose we take an element $e\in\Ff_{\varphi}$, by definition of this foliation $\Ff_{\varphi}$, we can write it as $e=\varphi \pi^*D$ for some $D\in TB$, so now we want to show
\[e^p=(\varphi\pi^* D)^p \overset{?}{=}\varphi \pi^* (D^p)\]
which would then imply $e^p\in\Ff_{\varphi}$ and thus implying the $p$-integrability of the foliation $\Ff_{\varphi}$. Since $\nabla$ is $p$-integrable, by definition the $p$-curvature $\psi_p(D)=0$,  i.e. 
\[\nabla_{D^p}s-(\nabla_D)^ps=0 
\text{ for all }s.\] 
Plugging in the formula for $\nabla$ in Proposition \ref{connection-splitting}, and by Lemma \ref{p-power}, for all $s$,
\[[\varphi\pi^*(D^p),\pi^*s]=\nabla_{D^p}s=(\nabla_D)^ps=[(\varphi\pi^*D)^p,\pi^*s].\]
Thus we have $e^p=(\varphi\pi^* D)^p =\varphi \pi^* (D^p)$. Thus $e^p\in \Ff_{\varphi}$, by definition we have $\Ff_{\varphi}$ is $p$-interable.\\ 

``$\nabla$ is $p$-integrable $\Leftarrow$ foliation $\Ff_{\varphi}$ is $p$-integrable'':

Suppose $e\in\Ff_{\varphi}\Rightarrow e^p\in\Ff_{\varphi}$, then we write $e=\varphi\pi^* D\in\Ff_{\varphi}$, and thus $e^p=(\varphi\pi^* D)^p\in\Ff_{\varphi}$, thus $(\varphi\pi^* D)^p =\varphi \pi^* (D^p)$. Thus we have for all $s$,
\[\nabla_{D^p}s=[\varphi\pi^*(D^p),\pi^*s]=[(\varphi\pi^*D)^p,\pi^*s]=(\nabla_D)^ps\]
Thus the p-curvature $\psi_p(D)=0$, i.e. the connection $\nabla$ is $p$-integrable.
\end{proof}

\begin{Coro}\label{equiv-conj-Coro}
Conjectures \ref{Conj-F} and \ref{p-curvature-conj} are equivalent.
\end{Coro}
\begin{proof}
This follows by combining Proposition \ref{algebraic-integrable-finite-monodromy} and Theorem \ref{p-integrability-equiv-connections-vs-foliation}. 
\end{proof}

\begin{remark}
In \cite[Theorem 2.2]{Bost-foliations}, Bost proved a conditional version of Conjecture F (Conjecture \ref{Conj-F}), with the additional assumption that the leaf (as an analytic manifold) satisfies the Liouville condition \cite[2.1.2]{Bost-foliations}. Thus by Corollary \ref{equiv-conj-Coro}, Bost also proved a conditional version of the Grothendieck-Katz $p$-curvature Conjecture assuming this Liouville condition. 
\end{remark}

\end{numberedparagraph}

\textit{Acknowledgements.} These notes were written while I was a first-year graduate student at Harvard, and last polished in September 2021. I am posting them on arXiv due to reader interest. I thank David Yang for helpful conversations.

\bibliographystyle{amsalpha}
\bibliography{bibfile}

\end{document}